\documentclass[11pt,a4paper,leqno]{amsart}

\usepackage[latin1]{inputenc}
\usepackage[T1]{fontenc}
\usepackage{amsfonts}
\usepackage{amsmath}
\usepackage{amssymb}
\usepackage{eurosym}
\usepackage{mathrsfs}
\usepackage{palatino}
\usepackage{color}
\usepackage{esint}
\usepackage{url}

 \usepackage{amsrefs}
 
\usepackage{hyperref} 
\hypersetup{
    colorlinks=false,       
    linkcolor=blue,          
    citecolor=magenta,        
    filecolor=magenta,      
    urlcolor=cyan           
}

\newcommand{\R}{\mathbb{R}}

\newcommand{\N}{\mathbb{N}}

\newcommand{\Z}{\mathbb{Z}}

\numberwithin{equation}{section}

\newcommand{\ave}[1]{\langle #1\rangle}

\newcommand{\BMO}[0]{\operatorname{BMO}}

\newcommand{\D}[0]{\mathcal{D}}

\swapnumbers
\theoremstyle{plain}
\newtheorem{thm}[equation]{Theorem}
\newtheorem{lem}[equation]{Lemma}
\newtheorem{prop}[equation]{Proposition}
\newtheorem{cor}[equation]{Corollary}

\theoremstyle{definition}

\newtheorem{exmp}[equation]{Example}

\theoremstyle{remark}

\author{Henri Martikainen}
\address[H.M.]{Department of Mathematics and Statistics, University of Helsinki, P.O.B. 68, FI-00014 Helsinki, Finland}
\email{henri.martikainen@helsinki.fi}
\thanks{Research of H.M. is supported by the Academy of Finland through the grant
Multiparameter dyadic harmonic analysis and probabilistic methods. }

\author{Tuomas Orponen}
\address[T.O.]{School of Mathematics, University of Edinburgh, James Clerk Maxwell Building, King's Buildings, Mayfield Road, Edinburgh, EH9 3JZ, Scotland}
\email{tuomas.orponen@helsinki.fi}
\thanks{Research of T.O. is supported by the Finnish foundation Jenny and Antti Wihurin Rahasto.}

\makeatletter
\@namedef{subjclassname@2010}{%
  \textup{2010} Mathematics Subject Classification}
\makeatother

\subjclass[2010]{42B20}
\keywords{Bi-parameter, paraproduct, unconditionality}

\title[Obstacles in characterising the boundedness of bi-parameter singular integrals]{Some obstacles in characterising the boundedness of bi-parameter singular integrals}

\begin{document}

\begin{abstract}
The famous $T1$ theorem for classical Calder\'on--Zygmund operators is a characterisation for their boundedness in $L^{2}$. In the bi-parameter case, on the other hand, the current
$T1$ theorem is merely a collection of sufficient conditions. This difference in mind, we study a particular
dyadic bi-parameter singular integral operator, namely the \emph{full mixed bi-parameter paraproduct} $P$, which is precisely the operator responsible for the outstanding problems
in the bi-parameter theory. We make several remarks about $P$, the common theme of which is to demonstrate the delicacy of the problem of finding a completely
satisfactory product $T1$ theorem. For example, $P$ need not be unconditionally bounded if it is conditionally bounded -- 
a major difference compared to the corresponding one-parameter model operators. Moreover, currently the theory even lacks a characterisation for the potentially easier unconditional boundedness. The product BMO condition is sufficient, but far from necessary: we show by example that unconditional boundedness does not even imply the weaker rectangular BMO condition.
\end{abstract}
\maketitle

\section{Introduction}
Our subject concerns bi-parameter Calder\'on--Zygmund theory, and the limitations it has compared to the one-parameter case.
A standard Calder\'on--Zygmund operator $T\colon L^2(\R^n) \to L^2(\R^n)$ can be viewed as an average of dyadic model operators (a result in this generality
proved by Hyt\"onen \cite{Hy}).
Among these dyadic model operators two special ones stand out: a dyadic paraproduct and a dual paraproduct associated with $T1$ and $T^*1$ respectively.
The rest are just very nice cancellative Haar shifts. The boundedness of a dyadic paraproduct is equivalent to the symbol belonging to BMO.
One way to establish the boundedness of $T$ is via the boundedness of these model operators. This then forces the assumption $T1, T^*1 \in \BMO$.
 But the one-parameter theory is unproblematic since if $T$ is bounded, then $T1$ and $T^*1$ belong to BMO. Indeed, the classical $T1$ theorem is a characterisation for the boundedness.
 Moreover, for all these one-parameter model operators boundedness is the same as unconditional boundedness.

Recently, the first named author proved a dyadic representation theorem for product singular integrals \cite{Ma}. The representation is more complicated than in the one-parameter case,
but the essential part in regard to the present paper is that the product BMO assumptions are tied to certain full paraproducts. Here "full" refers to the fact that also half paraproducts, which essentially have a paraproduct part
in only one of the parameters, appear. They do not concern us here, since they are not tied to the product BMO assumptions.
The full bi-parameter paraproducts come in two flavours: the standard one, and the mixed one (and their duals). The boundedness properties of the former operator are easy -- the boundedness is
characterised by the product BMO. The latter one is an evil twin of the first one: it is a very delicate operator and behind the limitations of the
product $T1$ theory. We aim to make this point very explicit in this note.

The limitation we refer to is the well-known handicap of the product theory that
the $T1$ theorem is just a collection of assumptions which guarantee the boundedness, but do not characterise it. Let us elaborate on this.
Following the above one-parameter story, one way to establish a product $T1$ is via the boundedness of the dyadic bi-parameter model operators.
The standard paraproducts are tied to the symbols $T1$ and $T^*1$. Like mentioned, the boundedness of the standard paraproduct is completely characterised by the product BMO condition.
This forces the assumption $T1, T^*1 \in \BMO_{\textup{prod}}$. Despite the complicated nature of the product BMO space, this condition is necessary for the boundedness
of $T$ -- this uses the famous covering lemma of Journ\'e \cite{Jo1}.

The mixed paraproduct (there is also the dual mixed paraproduct)
is tied to $T_1(1)$, where $T_1$ is the partial adjoint of $T$ i.e. $\langle T(f_1 \otimes g_1), f_2 \otimes g_2\rangle = \langle T_1(f_2 \otimes g_1), f_1 \otimes g_2\rangle$.
While $T_1$ is again a bi-parameter singular integral operator, the problem is that the boundedness of $T$ does not in general imply the boundedness of $T_1$.
So the assumption $T_1(1) \in \BMO_{\textup{prod}}$ is no longer a necessary condition for the boundedness of $T$. 
Nevertheless, this assumption is made and the known fact that the product BMO is a sufficient condition for the unconditional boundedness of the mixed paraproduct
is used. This means that the current product $T1$ theorems are a simultaneous characterisation for the boundedness of $T$ and $T_1$. One would, of course, prefer to characterise just the boundedness of $T$ alone.

Product BMO is sufficient for the boundedness of the mixed paraproduct, and because of the unconditional nature of the product BMO condition, it actually implies the \textbf{unconditional} boundedness.
But the gist is, as we will see, that the mixed paraproduct can be bounded, while being unconditionally unbounded. This implies that a BMO type
condition has no chance of characterising the boundedness of the mixed paraproduct. In fact, it turns out that product BMO is also overkill for unconditional boundedness: we construct an example showing that even the weaker rectangular BMO condition is not necessary for the unconditional boundedness of the mixed paraproduct. 

For harmonic analysis, singular integrals, and classical function spaces in the product (or multi-parameter) setting we refer to the works of Chang and Fefferman \cite{CF},
Fefferman \cite{Fe} and Fefferman and Stein \cite{FS}. As we have mentioned, the natural context for our results is offered by the boundedness criteria for product singular integrals. The first $T1$ theorem
in the product setting is due to Journ\'e \cite{Jo2}. Then there is the route via dyadic bi-parameter model operators \cite{Ma}.
These techniques were also extended to prove a non-homogeneous product $T1$ jointly with Hyt\"onen \cite{HM}.
For other related dyadic product methods and results we refer to the papers of Blasco and Pott \cite{BP}, Ou \cite{Ou}, Pipher and Ward \cite{PW}, and Treil \cite{Tr}.

Let us now explicitly introduce our object of study. It is the dyadic operator defined by the 
bilinear form
\begin{displaymath}
P_{\lambda}(f, g) := \sum_{I,J} \lambda_{IJ} \Big\langle f, h_I \otimes \frac{1_J}{|J|} \Big\rangle \Big\langle g, \frac{1_I}{|I|} \otimes h_J\Big\rangle,
\end{displaymath}
where $\lambda = (\lambda_{IJ})$ is an arbitrary sequence of reals indexed by dyadic rectangles $I \times J$, and $h_I$ denotes an $L^2$ normalised Haar function with zero mean.
The delicacy of finding an optimal product $T1$ theorem is already exhibited by this specific dyadic bi-parameter singular integral. On the other hand, results about this operator
can be transferred to continuous bi-parameter singular integrals via the representation theorem.

So, we wish to demonstrate that characterising the $L^{2}$ boundedness -- both conditional and unconditional -- of $P_{\lambda}$ in terms of the defining sequence $(\lambda_{IJ})_{IJ}$ alone is hard.
A very special class of paraproducts will already do for this task.
Indeed, we build our examples using \emph{space-independent paraproducts} i.e. those for which $\lambda_{IJ} = \lambda_{ij}$ if $(|I|,|J|) = (2^{-i},2^{-j})$. For these special paraproducts, we find that the $L^{2}$ problem is equivalent to the problem of characterising the $\ell^{2}$-boundedness of an arbitrary infinite matrix in terms of its entries -- to which we are not aware of any simple solution. The stated equivalence is the content of Theorem \ref{thm:mainNew} below.

We restrict attention to the notationally simplest case $\R \times \R$. Let $\mathcal{M}$ denote the space of all infinite matrices indexed by $\N \times \N$, and let $\Lambda$ denote the space of all sequences $(\lambda_{IJ})_{IJ}$ indexed by the dyadic rectangles $I \times J \subset \R \times \R$. Define a mapping $L \colon \mathcal{M} \to \Lambda$ by 
\begin{displaymath} L(A)_{IJ} = 2^{-(i+j)/2}A^{ij}, \quad \text{if } I \times J \subset [0,1)^{2} \text{ and } (|I|,|J|) = (2^{-i},2^{-j}). \end{displaymath}
For dyadic rectangles $I \times J \not\subset [0,1)^{2}$, set $L(A)_{IJ} = 0$.
\begin{thm}\label{thm:mainNew} The matrix $A$ is bounded on $\ell^{2}_{\N}$, if and only if the operator $P_{\lambda} = P_{L(A)}$ induced by the sequence $\lambda_{IJ} = L(A)_{IJ}$ is bounded on $L^{2}(\R^2)$. In fact, $\|A\|_{2 \to 2} = \|P_{L(A)}\|_{2 \to 2}$.
\end{thm}
Note that $P_{L(A)}$ is always space-independent. We list some corollaries.
\begin{cor} The problem of characterising the $L^{2}$-boundedness of a general paraproduct $P_{\lambda}$ in terms of the sequence $(\lambda_{IJ})_{IJ}$ is at least as difficult as the problem of characterising the $\ell^{2}$-boundedness of a general matrix $A \in \mathcal{M}$ in terms of its entries.
\end{cor}
Indeed, the general problem is at least as hard as the space-independent special case. A similar statement holds about characterising the \textbf{unconditional} boudedness of $P_{\lambda}$ in terms of \textbf{non-negative matrices}, simply by virtue of the fact that $L$ takes non-negative matrices to non-negative sequences. Combining Theorem \ref{thm:mainNew} with simple constructions, we can obtain the main examples:
\begin{cor}\label{cor:uc}
There exists a sequence $(\lambda_{IJ})_{IJ}$, such that $\|P_{\lambda}\|_{2 \to 2} < \infty$, but $\|P_{|\lambda|}\|_{2 \to 2} = \infty$. In other words, $P_{\lambda}$ is bounded, but not unconditionally bounded.
\end{cor}

\begin{cor}\label{cor:bmo}
There exists a non-negative sequence $(\lambda_{IJ})_{IJ}$, which is not in product $\BMO$ (in fact, not even rectangular $\BMO$), but nevertheless $\|P_{\lambda}\|_{2 \to 2} < \infty$. So, product BMO is not necessary for the unconditional boundedness of $P_{\lambda}$.
\end{cor}
\begin{exmp}
Let $\D$ be the standard dyadic grid in $\R$. Our construction gives us a sequence $\lambda = (\lambda_{IJ})_{I, J \in \D}$ so that $T = P_{\lambda}\colon L^2(\R^2) \to L^2(\R^2)$ boundedly
but $\tilde T = P_{|\lambda|}$ is not bounded on $L^2(\R^2)$. Now $T_1 = \Pi_{\lambda}$. This is not a bounded operator, since
$\|T_1\|_2 = \| \Pi_{\lambda}\|_2 \approx \|\lambda\|_{\textup{BMO}_{\textup{prod}}} = \infty$. This showcases the fact that the boundedness of a bi-parameter singular integral operator
does not, in general, imply the boundedness of its partial adjoint. But the implications for some possible generalisations of the current product $T1$ theorems are far more depressing than this.

Indeed, if there would be a general product $T1$ theorem, it would have to be such that
the assumptions one has to make about the numbers $\langle T_1(1), h_I \otimes h_J\rangle$ are so delicate that they do not imply the same condition for
$|\langle T_1(1), h_I \otimes h_J\rangle|$.
Notice that we have that $T1 = T^*1 = T_1^*(1) = 0 = \tilde T1 = \tilde T^*1 = \tilde T_1^*(1)$. We also have that $\langle T_1(1), h_I \otimes h_J\rangle = \lambda_{IJ}$
and that $T$ is bounded. But $\langle \tilde T_1(1), h_I \otimes h_J\rangle = |\lambda_{IJ}| = |\langle T_1(1), h_I \otimes h_J\rangle|$ and $\tilde T$ is unbounded.

A possible middle ground would be to impose a condition which characterises the unconditional boundedness of the mixed paraproduct. This condition certainly is not the product BMO.
Moreover, we have shown that a characterisation in terms of the sequence $(\lambda_{IJ})_{IJ}$ alone should not be expected. However, unconditionality allows for a reduction to positive bi-parameter
dyadic operators. Therefore, one can try to use (Sawyer type) testing conditions.
\end{exmp}

We record a few other results too. For example, in the course of proving Theorem \ref{thm:mainNew}, we obtain a natural "matrix-based" sufficient condition for the $L^{2}$-boundedness of $P_{\lambda}$, which is potentially useful even outside space-independent setting. But this only gives a general upper bound for $\|P_{\lambda}\|_{2 \to 2}$, and we will show that in general it is neither
weaker nor stronger than the product BMO. Finally, we demonstrate
that $\|P_{\lambda}\|_{2 \to 2} < \infty$ still implies some fairly strong conditions on the sequence $\lambda_{IJ}$. Indeed, some BMO conditions are allowed because of the fact that while it can be that
$\|P_{\lambda}\|_{2 \to 2} < \infty$ and $\|P_{|\lambda|}\|_{2 \to 2} = \infty$, the condition $\|P_{\lambda}\|_{2 \to 2} < \infty$ does imply that
$\|P_{\epsilon_p \lambda}\|_{2 \to 2} < \infty$ for \textbf{product signs} $\epsilon_p = (\epsilon_I \epsilon_J)$.

\section{Preliminaries}\label{preliminaries} 
\subsection{The spaces $X$ and $X'$}
Let $\D$ be the standard dyadic grid in the real line $\R$. We denote a general real sequence indexed by two dyadic cubes $I, J \in \D$ by
$\lambda = (\lambda_{IJ})$. In what follows $I$ and $J$ always stand for cubes from the dyadic grid $\D$. Sometimes we write $R = I \times J$ for a general dyadic rectangle. 

Let us define
\begin{displaymath}
\|\lambda\|_X := \sup \Big| \sum_{I, J} \lambda_{IJ} \ave{g_J}_I \ave{u_I}_J\Big|,
\end{displaymath}
where the supremum is taken over those function sequences $g_J, u_I\colon \R \to \R$ for which
$\#\{J\colon\, g_J \ne 0\} < \infty$, $\#\{I\colon \, u_I \ne 0\} < \infty$ and
\begin{displaymath}
\sum_{J} \|g_J\|_2^2 = \sum_{I} \|u_I\|_2^2 = 1.
\end{displaymath}
Let us define $X := \{\lambda\colon\, \|\lambda\|_X < \infty\}$. This is our space of conditional boundedness.
Similarly, let
\begin{displaymath}
\|\lambda\|_{X'} = \sup  \sum_{I, J} |\lambda_{IJ} \ave{g_J}_I \ave{u_I}_J|,
\end{displaymath}
and $X' := \{\lambda\colon\, \|\lambda\|_{X'} < \infty\}$. This is our space of unconditional boundedness.
One of our aims is to prove that $X \ne X'$. The open mapping theorem is useful in this task, so we first note that $X$ and $X'$ are complete:

\begin{lem}
X and X' are Banach spaces.
\end{lem}
\begin{proof}
Let us deal just with the space $X$. It is obvious that $X$ is a normed space. Now, let $(\lambda^N)_N$, where $\lambda^N = (\lambda^N_{IJ})$, be a Cauchy-sequence in $X$.
For fixed $I_0, J_0 \in \D$ the estimate $|\lambda_{I_0J_0}^N - \lambda_{I_0J_0}^M| \le |I_0|^{1/2} |J_0|^{1/2} \|\lambda^N - \lambda^M\|_X$ follows by letting
$g_J(x) = |I_0|^{-1/2}1_{I_0}(x)1_{J = J_0}(J)$ and $u_I(x) = |J_0|^{-1/2}1_{J_0}(x)1_{I = I_0}(I)$. Hence there exists $\lambda_{I_0J_0} = \lim_{N \to \infty} \lambda_{I_0J_0}^N$.
Let $\lambda = (\lambda_{IJ})$. Next, notice that
\begin{displaymath}
\| \lambda - \lambda^N\| \le \limsup_{M \to \infty} \|\lambda^M - \lambda^N\|_X \to 0,
\end{displaymath}
when $N \to \infty$. We conclude $\lambda \in X$ and $\lambda^N \to \lambda$ in $X$.
\end{proof}
\begin{lem}\label{lem:enough}
To prove that $X' \ne X$ it is enough to construct a sequence $(\lambda^N)_N$ so that $\|\lambda^N\|_{X'} \to \infty$, when $N \to \infty$, but
$\|\lambda^N\|_X \le 1$ for all $N$. 
\end{lem}
\begin{proof}
Notice that $\|\lambda\|_{X} \le \|\lambda\|_{X'}$.
If $X = X'$ then the open mapping theorem
implies that there is a $C < \infty$ so that $\|\lambda\|_{X'} \le C\|\lambda\|_X$ for all $\lambda$.
But this is impossible if a sequence $(\lambda^N)_N$ like in the statement of the lemma exists.
\end{proof}

\subsection{Connection to mixed paraproducts}
In this section, we establish the relevance of the spaces $X$ and $X'$ to the boundedness properties of mixed paraproducts. 
Let $h_I := |I|^{-1/2}(1_{I_l} - 1_{I_r})$, where $I_l, I_r \in \D$ are the left and right halves of $I$. That is, $h_I$ is an $L^2(\R)$ normalised Haar function with zero mean.
Let us define the dense collections $A, B \subset L^2(\R^2)$ by
\begin{displaymath}
A := \Big\{f\colon\, \|f\|_2= 1 \textup{ and } f = \sum_{I, J} f_{IJ}h_I \otimes h_J \textup{ with } f_{IJ} \ne 0 \textup{ for only finitely many } I\Big\}
\end{displaymath}
and
\begin{displaymath}
B := \Big\{g\colon\, \|g\|_2 = 1 \textup{ and } g = \sum_{I, J} g_{IJ}h_I \otimes h_J \textup{ with } g_{IJ} \ne 0 \textup{ for only finitely many } J\Big\}.
\end{displaymath}
We consider bilinear forms $P_{\lambda}$ initially defined for $f \in A$ and $g \in B$ by the formula
\begin{displaymath}
P_{\lambda}(f, g) := \sum_{I,J} \lambda_{IJ} \Big\langle f, h_I \otimes \frac{1_J}{|J|} \Big\rangle \Big\langle g, \frac{1_I}{|I|} \otimes h_J\Big\rangle.
\end{displaymath}
The summation in the definition of $P_{\lambda}(f,g)$ is finite, since $\Big\langle f, h_I \otimes \frac{1_J}{|J|} \Big\rangle \ne 0$ for only finitely many $I$ and
$\Big\langle g, \frac{1_I}{|I|} \otimes h_J\Big\rangle \ne 0$ for only finitely many $J$.
\begin{lem}
A sequence $(g_J) \in \ell^2(L^2(\R))$ is such that $\#\{J\colon\, g_J \ne 0\} < \infty$ and $\sum_{J} \|g_J\|_2^2 = 1$ if and only if
there exists a $g \in B$ such that for each $J$ we have $g_J = \langle g, h_J \rangle_2$, where  $\langle g, h_J\rangle_2(x) = \int_{\R} g(x,y)h_J(y)\,dy$.
\end{lem}
\begin{proof}
Given a sequence $(g_J)$ one defines $g$ by setting $g = \sum_J g_J \otimes h_J$. Then $g \in B$ and $g_J = \langle g, h_J \rangle_2$ for every $J$. The converse
direction is clear.
\end{proof}
Now notice that
\begin{displaymath}
\|P_{\lambda}\|_{2\to2} := \mathop{\sup_{f \in A}}_{g \in B} |P_{\lambda}(f,g)| = \mathop{\sup_{f \in A}}_{g \in B} \Big| \sum_{I,J} \lambda_{IJ} \ave{ \langle g, h_J\rangle_2}_I  \ave{\langle f, h_I\rangle_1}_J \Big| = \|\lambda\|_X.
\end{displaymath}
Let then $\epsilon = (\epsilon_{IJ})$, where $\epsilon_{IJ} = \pm 1$, and define $\epsilon \lambda = (\epsilon_{IJ} \lambda_{IJ})$. We have that
\begin{displaymath}
\sup_{\epsilon} \|P_{\epsilon \lambda}\|_{2\to2} = \sup_{\epsilon} \|\epsilon \lambda\|_X = \| \lambda \|_{X'}.
\end{displaymath}

\section{Main estimates and examples}\label{mainProofs}
As we have seen, for a given sequence $\lambda$, the $L^2(\R^2) \times  L^2(\R^2) \to \R$ boundedness of the bilinear form $P_{\lambda}$ is characterised by the $X$-norm of $\lambda$ i.e. the
$\ell^{2}_{\mathcal{D}}(L^{2}(\R)) \times \ell^{2}_{\mathcal{D}}(L^{2}(\R)) \to \R$ boundedness of the bilinear form
\begin{displaymath}
\Gamma_{\lambda}((u_{I}),(g_{J})) := \sum_{I,J} \lambda_{IJ} \langle g_{J} \rangle_{I} \langle u_{I} \rangle_{J}.
\end{displaymath}
We now introduce an operator $M_{\lambda}$, associated to the sequence $\lambda$ as follows. For $x,y \in \R$, consider the infinite matrix $M_{\lambda}(x,y)$ with the entries
\begin{displaymath}
M_{\lambda}^{ij}(x,y) := 2^{(i + j)/2} \lambda_{I_{i}(x)J_{j}(y)}, \qquad i,j \in \Z,
\end{displaymath}
where $I_{i}(x)$ and $J_{j}(y)$ are the unique dyadic intervals of lengths $2^{-i}$ and $2^{-j}$ containing $x$ and $y$, respectively. If $a = (a_{i})_{i \in \Z}$ is a sequence of functions
$\R^2 \to \R$, one may (formally) obtain another such sequence of functions $((M_{\lambda}a)_{j})_{j \in \Z}$ by considering the pointwise matrix product 
\begin{displaymath} (M_{\lambda}a)_{j}(x,y) := \sum_{i \in \Z} M_{\lambda}^{ij}(x,y)a_{i}(x,y). \end{displaymath} 

We now demonstrate a connection between the $\ell^{2}_{\mathcal{D}}(L^{2}(\R)) \times \ell^{2}_{\mathcal{D}}(L^{2}(\R)) \to \R$ boundedness of $\Gamma_{\lambda}$ and the
$\ell^2_{\Z}(L^2(\R^2)) \to \ell^2_{\Z}(L^2(\R^2))$ boundedness of $M_{\lambda}$. Let $\mathcal{D}_{i}$ denote the dyadic intervals of length $2^{-i}$. Then, we have that
\begin{align*}
|\Gamma_{\lambda}((u_{I}),(g_{J}))| & = \Big|\iint \sum_{I,J} \frac{\lambda_{IJ}}{|I||J|} g_{J}(x)1_{I}(x)u_{I}(y)1_{J}(y) \, dx \, dy\Big|\\
& = \Big|\iint \sum_{i,j \in \Z} 2^{i + j} \sum_{I \times J \in \mathcal{D}_{i} \times \mathcal{D}_{j}} \lambda_{IJ}g_{J}(x)1_{I}(x)u_{I}(y)1_{J}(y) \, dx \, dy\Big|\\
& = \Big| \iint \sum_{i,j \in \Z} M_{\lambda}^{ij}(x,y) [2^{i/2}u_{I_{i}(x)}(y)] [2^{j/2}g_{J_{j}(y)}(x)] \, dx \, dy \Big|\\
& = |\langle M_{\lambda}a, b \rangle_{\ell^2_{\Z}(L^2(\R^2))} |, 
\end{align*}
where $a = (a_i)$, $b=(b_j)$ are given by
\begin{displaymath} a_{i}(x,y) = 2^{i/2}u_{I_{i}(x)}(y) \quad \text{and} \quad b_{j}(x,y) = 2^{j/2}g_{J_{j}(y)}(x). \end{displaymath}
There holds that
\begin{displaymath}
\|a\|_{\ell^2_{\Z}(L^2(\R^2))} = \|(u_{I})\|_{\ell^{2}_{\mathcal{D}}(L^{2}(\R))} \quad \text{and} \quad \|b\|_{\ell^2_{\Z}(L^2(\R^2))} = \|(g_{J})\|_{\ell^{2}_{\mathcal{D}}(L^{2}(\R))}.
\end{displaymath}
This proves the following proposition:
\begin{prop}[The $M_{\lambda}$ condition]\label{Mlambda} For any sequence $\lambda = (\lambda_{IJ})$, we have that
\begin{equation}\label{form1}
\|\lambda\|_{X} \leq \|M_{\lambda}\|_{\ell^2_{\Z}(L^2(\R^2)) \to \ell^2_{\Z}(L^2(\R^2))} =: \|M_{\lambda}\|.
\end{equation}
\end{prop}
In general, the condition $\|M_{\lambda}\| < \infty$ is sufficient but not necessary for the boundedness of $P_{\lambda}$ -- we will demonstrate this by an example later.
However, when specialising to space-independent paraproducts we have the equality of Theorem \ref{thm:mainNew}.
\begin{proof}[Proof of Theorem \ref{thm:mainNew}]
Let a matrix $A \in \mathcal{M}$ be given. We have $\|P_{L(A)}\|_{2 \to 2} = \|L(A)\|_X \le \|M_{L(A)}\| \le \|A\|$. The last inequality follows from $M_{L(A)}(x,y) = A$ for $(x,y) \in [0,1)^{2}$.
Conversely, given two sequences of numbers $(a_i)_{i \in \N}$ and $(b_j)_{j \in \N}$ we have that
\begin{align*}
\Big| \sum_{i, j \in \N} A^{ij} a_i b_j\Big| &= \Big| \sum_{i, j \in \N} A^{ij} a_i b_j 2^{-(i+j)} \mathop{\sum_{I \times J \in \mathcal{D}_{i} \times \mathcal{D}_{j}}}_{I \times J \subset [0,1)^2} 1 \Big| \\
&= \Big| \sum_{i, j \in \N} \mathop{\sum_{I \times J \in \mathcal{D}_{i} \times \mathcal{D}_{j}}}_{I \times J \subset [0,1)^2} L(A)_{IJ} 2^{-j/2}b_j 2^{-i/2}a_i\Big|
= \Big| \sum_{I, J} L(A)_{IJ} \langle g_J \rangle_I \langle u_I \rangle_J \Big|
\end{align*}
for $g_{J} = 2^{-j/2}b_{j}1_{[0,1)}$, if $|J| = 2^{-j}$, $j \in \N$, and $u_{I} = 2^{-i/2}a_{i}1_{[0,1)}$, if $|I| = 2^{-i}$, $i \in \N$. 
Noting that e.g. $\|(b_{j})\|_{\ell^{2}_{\N}} = \|(g_{J})\|_{\ell^{2}_{\mathcal{D}}(L^{2}(\R))}$ we see that $\|A\| \le \|L(A)\|_X = \|P_{L(A)}\|_{2 \to 2}$.
So we conclude that $\|A\| = \|P_{L(A)}\|_{2 \to 2}$.
\end{proof}

Next, we record the proofs of our main examples i.e. the proofs of Corollaries \ref{cor:uc} and \ref{cor:bmo}. We start from the latter since it is simpler.
\begin{proof}[Proof of Corollary \ref{cor:bmo}]
We recall that
\begin{displaymath}
\|\lambda\|_{\BMO_{\textup{rec}}} := \sup_{I_0 \times J_0 \in \D \times \D} \Big( \frac{1}{|I_0 \times J_0|} \sum_{I, J:\, I \times J \subset I_0 \times J_0} |\lambda_{IJ}|^2 \Big)^{1/2}.
\end{displaymath}
Define the non-negative sequence $\lambda_{IJ} = |I|^{1/2}|J|^{1/2}$, if $|I| = |J|$ and $I \times J \subset [0,1)^{2}$. Otherwise set $\lambda_{IJ} = 0$. Then, there holds that
\begin{displaymath}
\|\lambda\|_{\BMO_{\textup{rect}}}^2 \ge \sum_{I \times J \subset [0,1)^{2}} \lambda_{IJ}^{2} = \sum_{I \subset [0,1)} \sum_{J : |J| = |I|} |I||J| = \sum_{I \subset [0,1)} |I| = \infty.
\end{displaymath}
However, now $L(\operatorname{Id}) = \lambda$, where $\operatorname{Id}$ is the identity matrix.
Also using the fact that the numbers $\lambda_{IJ}$ are non-negative, we infer that
$\|\lambda\|_{X'} = \|\lambda\|_X = \|\operatorname{Id}\| = 1$. Therefore, $P_{\lambda}$ is unconditionally bounded but $\lambda$ does not belong to the product BMO.
\end{proof}

\begin{proof}[Proof of Corollary \ref{cor:uc}]
Let us fix $N = 2^m - 1$ for some $m \in \N$. We shall consider certain specific signs $\epsilon_{ij} = \pm 1$, $i,j \in \{0, 1, \ldots, N\}$, known as the Fourier basis in the area of Fourier analysis on the discrete unit cube.
Let us explain what these are.
They are signs for which the vectors $v_i = (\epsilon_{ij})_{j=0}^N \in \R^{N+1}$, $i = 0, 1, \ldots, N$, satisfy the orthogonality relation $v_i \cdot v_{i'} = (N+1)\delta_{ii'}$. Given these signs we may define
$\lambda^N = L(A_N)$, where $A_N^{ij} = (N+1)^{-1/2}\epsilon_{ij}$ for $i,j \in \{0, 1, \ldots, N\}$ and $A_N^{ij} = 0$ otherwise. Now $\|\lambda^N\|_{X'} = \||\lambda^N|\|_X \ge (N+1)^{1/2}$
but $\|\lambda^N\|_X \le 1$ by Theorem \ref{thm:mainNew}. The proof is finished by appealing to Lemma \ref{lem:enough}.

We still note that should the reader not be familiar with the Fourier basis, one does not really need to use these specific signs.
Indeed, a basic theorem from random matrix theory
implies that with independent choice of signs $\pm 1$ the norm of the corresponding $(N+1) \times (N+1)$-matrix is approximately $\sqrt{N+1}$ with high probability (for all large $N$).
For example, see Corollary 2.3.5 of Tao's book \cite{Ta}.
\end{proof}
\section{Additional remarks and examples}
\subsection{The $M_{\lambda}$ condition and product BMO}
We study the connection of the $M_{\lambda}$ condition with the classical product BMO condition, which is known to be sufficient even for unconditional boundedness of $P_{\lambda}$.
We start by giving a short proof of this fact. Recall that
\begin{equation*}
\|\lambda\|_{\BMO_{\textup{prod}}} := \sup_{\Omega} \Big( \frac{1}{|\Omega|} \sum_{I, J:\, I \times J \subset \Omega} |\lambda_{IJ}|^2 \Big)^{1/2},
\end{equation*}
where the supremum is taken over those sets $\Omega \subset \R^2$ such that $|\Omega| < \infty$ and such that for every $x \in \Omega$ there exists $I, J$ so that $x \in I \times J \subset \Omega$.
\begin{prop}
There holds that
\begin{displaymath}
\|\lambda\|_{X'} \lesssim \|\lambda\|_{\BMO_{\textup{prod}}}.
\end{displaymath}
\end{prop}
\begin{proof}
Consider another sequence $A = (A_{IJ})$.
We will show that
\begin{equation}\label{eq:H}
\sum_{I, J} |\lambda_{IJ}| |A_{IJ}| \lesssim \|\lambda\|_{\BMO_{\textup{prod}}} \|s_A\|_1, \qquad s_A := \Big( \sum_{I, J} |A_{IJ}|^2 \frac{1_{I \times J}}{|I \times J|} \Big)^{1/2}.
\end{equation}
Indeed, let $U_k := \{s_A > 2^k\}$, $V_k := \{M1_{U_k} > 1/2\}$ and $\mathcal{R}_k := \{R = I \times J\colon \, |R \cap U_k| > |R|/2\}$. 
Here $M$ is the (dyadic) strong maximal function in $\R^2$. We make a sequence of useful observations.
Notice first that $\bigcup_{R \in \mathcal{R}_k} R \subset V_k$. If $R \not \in \bigcup_{k \in \Z} \mathcal{R}_k$, then $|R \cap \{s_A = 0\}| \ge |R|/2$. Therefore, we have that
$$
\sum_{R:\, R \not \in \bigcup_{k \in \Z} \mathcal{R}_k} |A_R|^2 \le 2\int_{\{s_A = 0\}} s_A^2 = 0.
$$
So if $A_R \ne 0$ then $R \in \bigcup_{k \in \Z} \mathcal{R}_k$. If $R \in \bigcap_{k \in \Z} \mathcal{R}_k$, then $0 = |R \cap \{s_A = \infty\}| \ge |R|/2$ (we may assume
$\|s_A\|_1 < \infty$) -- but this is absurd.

We may thus rearrange the summation and then estimate as follows:
\begin{align*}
\sum_{R = I \times J} |\lambda_R| |A_R| &= \sum_{k \in \Z} \sum_{R: \, R \in \mathcal{R}_k \setminus \mathcal{R}_{k+1}} |\lambda_R| |A_R| \\
&\le \sum_{k \in \Z} \Big( \sum_{R: \, R \subset V_k} |\lambda_R|^2\Big)^{1/2} \Big( 2 \int_{V_k \setminus U_{k+1}} s_A^2 \Big)^{1/2} \\
&\lesssim \|\lambda\|_{\BMO_{\textup{prod}}} \sum_{k \in \Z} 2^k|V_k| \lesssim  \|\lambda\|_{\BMO_{\textup{prod}}} \sum_{k \in \Z} 2^k|U_k| \approx \|\lambda\|_{\BMO_{\textup{prod}}} \|s_A\|_1.
\end{align*}
We have proved \eqref{eq:H}. Let us then specify $A_{IJ} :=  \ave{g_J}_I \ave{u_I}_J$. Let $M_{\D}$ denote the dyadic maximal function in $\R$. With this choice we have that
$$
s_A \le  \Big( \sum_J [M_{\D}(g_J)]^2 \otimes \frac{1_J}{|J|}\Big)^{1/2} \Big( \sum_I \frac{1_I}{|I|} \otimes [M_{\D}(u_I)]^2\Big)^{1/2},
$$
and therefore
\begin{align*}
\|s_A\|_1 &\le  \Big( \sum_J \|M_{\D}(g_J)\|_2^2 \Big)^{1/2}  \Big( \sum_I \|M_{\D}(u_I)\|_2^2 \Big)^{1/2}
&\lesssim \Big( \sum_J \|g_J\|_2^2 \Big)^{1/2}  \Big( \sum_I \|u_I\|_2^2 \Big)^{1/2}.
\end{align*} \end{proof}

The example below portraits a sequence $(\lambda_{IJ})_{IJ}$, which is in product BMO but fails to satisfy the $M_{\lambda}$ condition. In light of the previous proposition, this implies that the $M_{\lambda}$ condition is \textbf{not} necessary for the (conditional or unconditional) boundedness of $P_{\lambda}$.
\begin{exmp}
Define 
\begin{displaymath} \lambda_{[0,1] \times [0,2^{-j}]} := 2^{-j/2}, \qquad j \in \N, \end{displaymath}
and $\lambda_{I \times J} = 0$, if $I \times J$ is not of the form $[0,1] \times [0,2^{-j}]$. Then the product BMO condition is satisfied: indeed, if $\Omega \subset \R^{2}$ is an open set, and $[0,1] \times [0,2^{-k}]$ is the largest rectangle of the form $[0,1] \times [0,2^{-j}]$ contained in $\Omega$, then
\begin{displaymath} \sum_{R \subset \Omega} \lambda_{R}^{2} \lesssim \lambda_{[0,1] \times [0,2^{-k}]}^{2} = 2^{-k} = [0,1] \times [0,2^{-k}] \leq |\Omega|. \end{displaymath}

On the other hand, $\|M_{\lambda}\| = \infty$. To see this, we pick a sequence of functions $a = (a_{i})_{i \in \Z} \colon \R^{2} \to \R$ such that $a_{i} \equiv 0$ for $i \ne  0$, and, say,
\begin{displaymath} a_{0}(x,y) = \frac{1_{[0,1]^{2}}(x,y)}{\sqrt{|y|} \cdot (1 - \ln |y|)}. \end{displaymath}
Then, we have that
\begin{align*} \|M_{\lambda}a\|_{\ell^{2}_{\Z}(L^{2})}^{2} & = \iint \sum_{j \in \Z} \Big( \sum_{i \in \Z} 2^{(i + j)/2}\lambda_{I_{i}(x)J_{j}(y)}a_{i}(x,y) \Big)^{2} \, dx \, dy\\
& = \iint \sum_{j = 0}^{\infty} \Big(2^{j/2}\lambda_{[0,1] \times J_{j}(y)}a_{0}(x,y) \Big)^{2} \, dx \, dy\\
& = \int_{0}^{1} \sum_{j = 0}^{\infty} 2^{j}\lambda_{[0,1] \times J_{j}(y)}^{2}\frac{1}{y \cdot (1 - \ln y)^{2}} dy =: \int_{0}^{1} \frac{c(y) \, dy}{y \cdot (1 - \ln y)^{2}}.  \end{align*} 
For $y \in [0,1)$, the quantity $c(y)$ counts the number of rectangles $[0,1] \times [0,2^{-j}]$ containing $y$, so that $c(y) \sim 1-\ln y$. Consequently,
\begin{displaymath} \|M_{\lambda} a\|_{\ell_{\Z}^{2}(L^{2})}^{2} \sim \int_{0}^{1} \frac{dy}{y \cdot (1 - \ln y)} = \infty, \end{displaymath}
even though $\|a\|_{\ell^{2}_{\Z}(L^{2})} = \|a_{0}\|_{2} < \infty$.
\end{exmp}

\subsection{Necessary conditions for boundedness}

So far, we have only seen examples of conditions, which the $L^{2}$ boundedness of $P_{\lambda}$ does \textbf{not} imply.
Let us balance the scales a bit by recording that $\|P_{\lambda}\|_{2 \to 2} < \infty$ is still a fairly strong requirement for the sequence $\lambda$. 
\begin{prop}\label{onePar} The following BMO condition holds
\begin{equation}\label{I0}
\sup_{I_0 \in \D} \sup_{J_0 \in \D} \Big( \frac{1}{|I_0| |J_0|} \sum_{J \subset J_{0}} \lambda_{I_{0}J}^{2} \Big)^{1/2} \le \|\lambda\|_X.
\end{equation}
The symmetric condition also holds.
\end{prop}
\begin{proof}
Simply choosing $u_{I}(x) = 1_{J_{0}}(x)1_{I = I_{0}}(I)$ and $g_{J}(x) = \lambda_{I_{0}J}1_{I_{0}}(x)1_{J \subset J_{0}}(J)$ we see that
\begin{displaymath}
\sum_{J \subset J_{0}} \lambda_{I_{0} J}^{2} = \sum_{I, J} \lambda_{IJ} \langle g_{J} \rangle_{I}\langle u_{I} \rangle_{J} \le \|\lambda\|_X \Big( \sum_{J \subset J_{0}} |I_{0}|\lambda_{I_{0} J}^{2} \Big)^{1/2}|J_{0}|^{1/2}.
\end{displaymath}
\end{proof}

The finiteness of the left hand side of \eqref{I0} is characterised by the (unconditional) inequality:
\begin{displaymath}
\sum_{I, J} |\lambda_{IJ}\langle g_{J} \rangle_{I} \langle u_{I} \rangle_{J}| \lesssim \Big(\sum_{J} \|g_{J}\|_{2}^2\Big)^{1/2}\Big( \sum_{I} \|u_{I}\|_{2} \Big).
\end{displaymath}
To see the sufficiency of \eqref{I0}, first apply Cauchy-Schwarz in the $J$-summation and then estimate the remaining expressions by the Carleson embedding theorem. In conclusion, the assumption $\|P_{\lambda}\|_{2 \to 2} < \infty$ is not strong enough to imply that the bilinear form $\Gamma_{\lambda}$ (as in Section \ref{mainProofs}) is unconditionally bounded $\ell^{2}(L^{2}) \times \ell^{2}(L^{2}) \to \R$, but it is strong enough to imply that $\Gamma_{\lambda}$ is unconditionally bounded $\ell^{1}(L^{2}) \times \ell^{2}(L^{2}) \to \R$ and $\ell^{2}(L^{2}) \times \ell^{1}(L^{2}) \to \R$.
\subsection*{Acknowledgements}
H.M. would like to thank Michael Lacey and the School of Mathematics of Georgia Tech, where H.M. conducted this research as a visitor, for hospitality.
We thank Michael Lacey for suggesting this theme of study and for useful discussions.

\end{document}